\documentclass[11pt]{amsart}
\usepackage{graphicx}
\usepackage{amssymb}
\usepackage{amscd}
\usepackage{cite}
\usepackage[all]{xy}
\usepackage{multicol}
\usepackage{url}
\newtheorem{tw}{Theorem}[section]
\newtheorem{prop}[tw]{Proposition}

\theoremstyle{remark}
\newtheorem{uw}[tw]{Remark}

\theoremstyle{definition}

\newcommand{\Cal}[1]{\mathcal{#1}}
\newcommand{\bez}{\setminus}
\newcommand{\podz}{\subseteq}

\newcommand{\eps}{\varepsilon}

\newcommand{\kre}[1]{\overline{#1}}
\newcommand{\gen}[1]{\langle #1 \rangle}
\newcommand{\map}[3]{#1\colon #2\to #3}
\newcommand{\field}[1]{\mathbb{#1}}
\newcommand{\zz}{\field{Z}}

\newcommand{\st}{\;|\;}

\newcommand{\SR}[1]{\underset{\rightarrow}{[#1]}}
\newcommand{\SL}[1]{\underset{\leftarrow}{[#1]}}
\newcommand{\ukre}[1]{\underline{#1}}

\makeatletter
\@namedef{subjclassname@2020}{\textup{2020} Mathematics Subject Classification}
\makeatother

\begin{document}

\numberwithin{equation}{section}
\title[The first homology group with twisted coefficients \ldots]
{The first homology group with twisted coefficients for the mapping class group of a non--orientable surface of genus three with two boundary components }

\author{Piotr Pawlak \hspace{1em} Micha\l\ Stukow}

\address[]{
Institute of Mathematics, Faculty of Mathematics, Physics and Informatics, University of Gda\'nsk, 80-308 Gda\'nsk, Poland }

\email{piotrus.pawlak@wp.pl, trojkat@mat.ug.edu.pl}




\begin{abstract}
We determine the first homology group with coefficients in $H_1(N;\zz)$ for the mapping class group 
of a non--orientable surface $N$ of genus three with two boundary components.
\end{abstract}

\maketitle%
 \section{Introduction}%
Let $N_{g,s}^n$ be a smooth, non--orientable, compact surface of genus $g$ with $s$ boundary components and $n$ punctures. If $s$ and/or $n$ is zero, then we omit it from the notation. If we do not want to emphasise the numbers $g,s,n$, we simply write $N$ for a surface $N_{g,s}^n$. Recall that $N_{g}$ is a connected sum of $g$ projective planes and $N_{g,s}^n$ is obtained from $N_g$ by removing $s$ open discs and specifying a set $\Sigma=\{z_1,\ldots,z_n\}$ of $n$ distinguished points in the interior of~$N$.

Let ${\textrm{Diff}}(N)$ be the group of all diffeomorphisms $\map{h}{N}{N}$ such that $h$ is the identity 
on each boundary component. By ${\Cal{M}}(N)$ we denote the quotient group of ${\textrm{Diff}}(N)$ by
the subgroup consisting of maps isotopic to the identity, where we assume that isotopies are 
the identity on each boundary component. ${\Cal{M}}(N)$ is called the \emph{mapping class group} of $N$. 

Let ${\Cal{PM}}^+(N)$ be the subgroup of ${\Cal{M}}(N)$ consisting of elements which fix $\Sigma$ pointwise and preserve a local orientation around the punctures $\{z_1,\ldots,z_k\}$.

\subsection{Main results}
We showed in \cite{Stukow_homolTopApp} that if  $N_{g,s}$ is a non--orientable surface of genus $g\geq 3$ with $s\leq 1$ boundary components, then
\begin{equation*}H_1({\Cal{M}}(N_{g,s});H_1(N_{g,s};\zz))\cong \begin{cases}
                                                 \zz_2\oplus\zz_2\oplus\zz_2 &\text{if $g\in\{3,4,5,6\},$}\\
                                                 \zz_2\oplus\zz_2&\text{if $g\geq 7$.}
                                                \end{cases}
\end{equation*}
It is quite natural to extend this result to surfaces with more than one boundary component. However, this is quite challenging, because we do not have a simple presentation for the mapping class group in this case. Nevertheless,
in the forthcoming paper \cite{PawlakStukowHomoPunctured} we managed (among other homological results) to compute the first homology group $H_1({\Cal{M}}(N_{g,s});H_1(N_{g,s};\zz))$ for any $s\geq 0$ by bounding it both from above and below. However, it turned out that the proof of nontriviality of two important families of homology generators (denoted by $a_{1,3+j}$ and $u_{1,3+j}$ in \cite{PawlakStukowHomoPunctured}) is beyond the methods used in 
\cite{PawlakStukowHomoPunctured}. By a simple reduction, the required argument in \cite{PawlakStukowHomoPunctured} boils down to the exceptional case of a non-orientable surface $N_{3}^2$ of genus 3 with two punctures. The main purpose of this paper is to fill this gap, and to prove the following two theorems
\begin{tw}\label{MainThm1}
 If  $N_{3,2}$ is a non--orientable surface of genus $g= 3$ with 2 boundary components, then 
 \[H_1({\Cal{M}}(N_{3,2});H_1(N_{3,2};\zz))\cong \zz_2^{6}.\]
\end{tw}
\begin{tw}\label{MainThm2}
 If  $N_{3}^2$ is a non--orientable surface of genus $g= 3$ with 2 punctures, then 
 \[H_1({\Cal{PM}^+}(N_{3}^2);H_1(N_{3}^2;\zz))\cong \zz_2^{5}.\]
\end{tw}
\subsection{Homology of groups}
Let us briefly review how to compute the first homology of a group with twisted coefficients -- for more details see Section~5 of \cite{Stukow_HiperOsaka} and references there. 

For a given group $G$ and $G$-module $M$ (that is $\zz G$-module) we define $C_2(G)$ and $C_1(G)$ as the free $G$-modules generated
respectively by symbols $[h_1|h_2]$ and $[h_1]$, where $h_i\in G$. We define also $C_0(G)$ as the free $G$-module generated by the
empty bracket $[\cdot]$. Then the first homology group $H_1(G;M)$ is the first homology group of the complex
\[\xymatrix@C=3pc@R=3pc{C_2(G)\otimes_G M\ar[r]^{\ \partial_2\otimes_G {\rm id}}&C_1(G)\otimes_G M
\ar[r]^{\ \partial_1\otimes_G {\rm id}}&C_0(G)\otimes_G M},\]
where 
\begin{equation}\label{diff:formula}
\begin{aligned}
\partial_2([h_1|h_2])&=h_1[h_2]-[h_1h_2]+[h_1],\\
   \partial_1([h])&=h[\cdot]-[\cdot].
  \end{aligned} 
\end{equation}
For simplicity, we write $\otimes_G=\otimes$ and $\partial\otimes {\rm id}=\kre{\partial}$ henceforth.

If the group $G$ has a presentation $G=\langle X\,|\,R\rangle$ and
\[\gen{\kre{X}}=\gen{[x]\otimes m\st x\in X, m\in M}\podz C_1(G)\otimes M,\]
then $H_1(G;M)$ is a quotient of  $\gen{\kre{X}}\cap \ker\kre{\partial}_1$.

The kernel of this quotient corresponds to relations in $G$
(that is elements of $R$). To be more precise, if
$r\in R$ has the form $x_1\cdots x_k=y_1\cdots y_n$ and $m\in M$, then 
$r$ gives the relation (in $H_1(G;M)$)
\begin{equation}
 \kre{r}\otimes m\!:\ \sum_{i=1}^{k}x_1\cdots x_{i-1}[x_i]\otimes m=\sum_{i=1}^{n}y_1\cdots y_{i-1}[y_i]\otimes m.\label{eq_rew_rel}
\end{equation}
Then 
\[H_1(G;M)=\gen{\kre{X}}\cap \ker\kre{\partial}_1/\gen{\kre{R}},\]
where 
\[\kre{R}=\{\kre{r}\otimes m\st r\in R,m\in M\}.\]
 \section{Presentations for the groups ${\Cal{PM}}^{+}(N_{3}^2)$ and ${\Cal{M}}(N_{3,2})$}
Represent the surface $N_{3,2}$ as a sphere with three crosscaps $\mu_1,\mu_2,\mu_3$ and two boundary components (Figure \ref{r01}). Let 
\[\alpha_1,\alpha_{2},\eps_1,\eps_2, \delta_1,\delta_2,\beta_1,\beta_2,\beta_3\]
be two--sided circles indicated in Figure \ref{r01}. 
\begin{figure}[h]
\begin{center}
\includegraphics[width=0.88\textwidth]{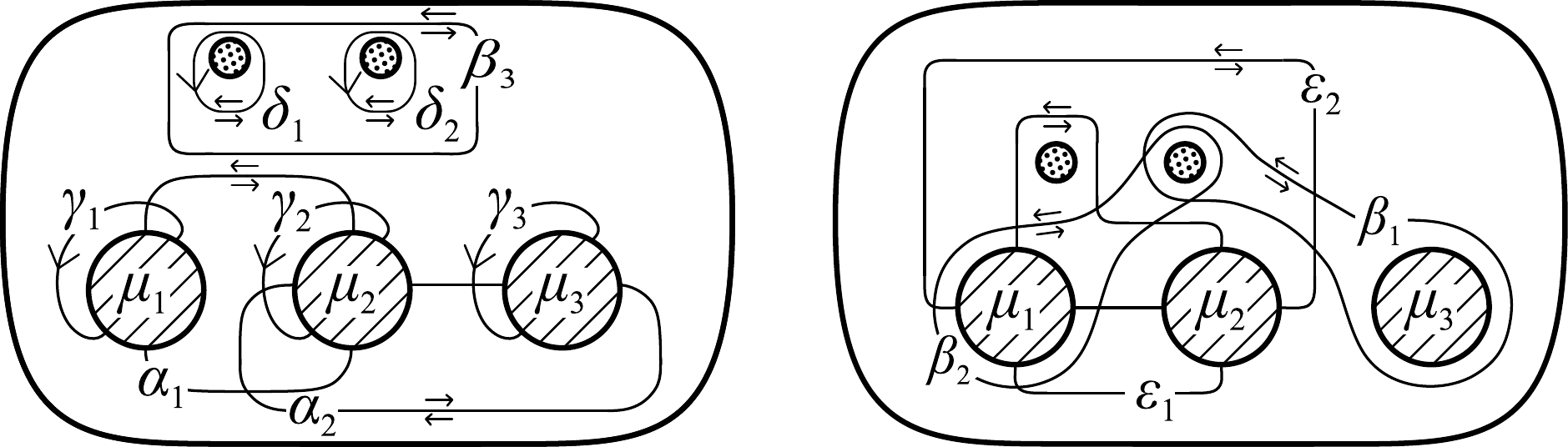}
\caption{Surface $N_{3,2}$ as a sphere with three crosscaps.}\label{r01} %
\end{center}
\end{figure}
Small arrows in that figure indicate directions of Dehn twists 
\[a_1,a_{2},e_1,e_2,d_1,d_2, b_1,b_2,b_3
\] associated with these circles. 

Let $u$ be a \emph{crosscap transposition}, that is the map which interchanges the first two crosscaps (see Figure \ref{r03}). 
\begin{figure}[h]
\begin{center}
\includegraphics[width=0.62\textwidth]{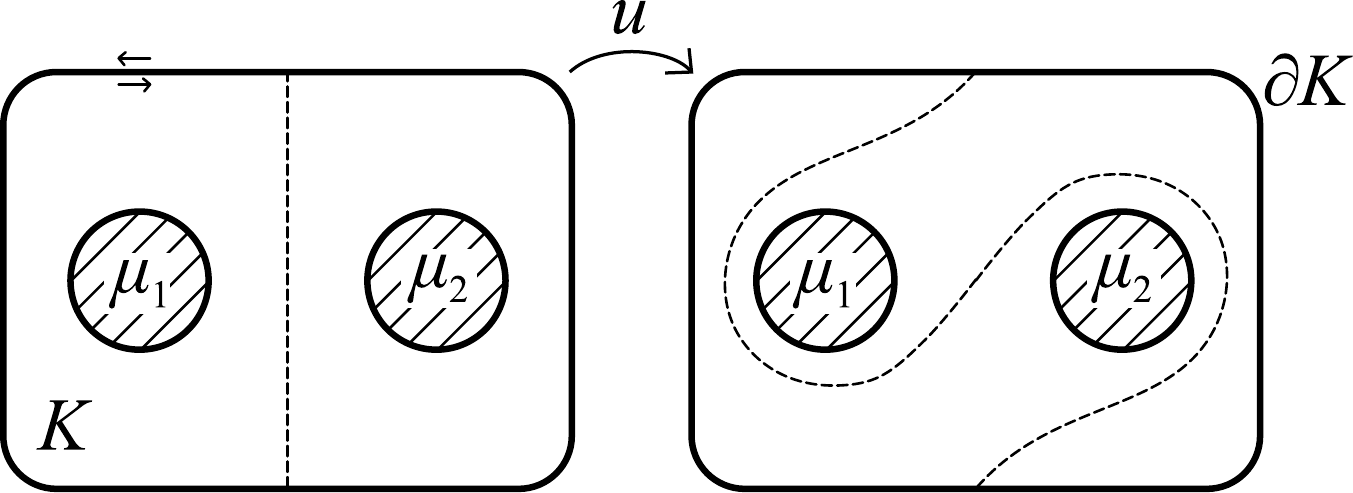}
\caption{Crosscap transposition $u$.}\label{r03} %
\end{center}
\end{figure}
\begin{tw}\label{tw:pres:two:holes}
 The mapping class group ${\Cal{M}}(N_{3,2})$ of a non--orientable surface of genus 3 with two boundary components admits a presentation with generators $\{a_1,a_2,e_1,e_2,d_1,d_2,b_1,b_2,b_3, u\}$. The defining relations are
 \begin{multicols}{2}
 \begin{itemize}
  \item[(1a)] $a_1e_1=e_1a_1$, 
  \item[(1b)] $a_1e_2=e_2a_1$,
  \item[(1c)] $e_1e_2=e_2e_1$,
  \item[(2a)] $a_1a_2a_1=a_2a_1a_2$,
  \item[(2b)] $e_1a_2e_1=a_2e_1a_2$,
  \item[(2c)] $e_2a_2e_2=a_2e_2a_2$,
  \item[(3)] $a_1ua_1=u$,
  \item[(4)] $ue_2=a_2ue_2a_2$,
  \item[(5)] $ub_1=b_1u$,
  \item[(6)] $ub_3=b_3u$,
  \item[(7)] $a_2b_2=b_2a_2$,
  \item[(8)] $(e_1u)^2=d_1b_1$,
  \item[(9a)] $b_3=(e_2u)^2$,
  \item[(9b)] $(e_2u)^2=(a_2e_2a_1^2)^3$,
  \item[(10)] $ue_1u^{-1}b_2u=a_2b_1^{-1}a_2^{-1}ue_1$,
  \item[(11)] $b_3b_1(b_2u)^2=d_2^2d_1u^2$,
  \item[(12)] $(a_2e_2e_1a_1)^3=d_1d_2$,
  \item[(13a)] $d_1a_1=a_1d_1$, 
  \item[(13b)] $d_1a_2=a_2d_1$,
  \item[(13c)] $d_1e_1=e_1d_1$,
  \item[(13d)] $d_1u=ud_1$, 
  \item[(13e)] $d_2u=ud_2$. 
 \end{itemize}
 \end{multicols}
\end{tw}
\begin{proof}
 By Theorem 7.18 of \cite{Szep_curv}, the mapping class group ${\Cal{M}}(N_{3,2})$ admits a presentation with generators $\{A_1,A_2,A_3,B,D_1,D_2,D_3,U, C_1,C_2\}$ and relations
 \begin{itemize}
  \item[(S1)] $A_iA_j=A_jA_i$, $i,j=1,2,3$, 
  \item[(S2)] $A_iBA_i=BA_iB$, $i=1,2,3$,
  \item[(S3)] $UA_1U^{-1}=A_1^{-1}$,
  \item[(S4)] $UBU^{-1}=A_3^{-1}B^{-1}A_3$,
  \item[(S5)] $UD_1=D_1U$,
  \item[(S6)] $UD_3=D_3U$,
  \item[(S7)] $BD_2=D_2B$,
  \item[(S8)] $(UA_2)^2=D_1C_1$,
  \item[(S9)] $(A_1^2A_3B)^3=(UA_3)^2=D_3$,
  \item[(S10)] $A_2^{-1}UD_2U^{-1}A_2=UB^{-1}D_1^{-1}BU^{-1}$,
  \item[(S11)] $(UD_2)^2D_1D_3=U^2C_1C_2^2$,
  \item[(S12)] $(A_1A_2A_3B)^3=C_1C_2=C_2C_1$,
  \item[(S13)] $C_iA_j=A_jC_i$, $C_iD_j=D_jC_i$, $C_iB=BC_i$, $C_iU=UC_i$, $i=1,2$, $j=1,2,3$.
 \end{itemize}
Moreover, the topological model for the surface $N_{3,2}$ can be chosen in such way that
\[\begin{aligned}
A_1&=a_1^{-1},\ A_2=e_1^{-1},\ A_3=e_2^{-1},\\
B&=a_2^{-1},\ U=u^{-1},\\
D_1&=b_1^{-1},\ D_2=b_2^{-1},\ D_3=b_3^{-1},\\
C_1&=d_1^{-1},\ C_2=d_2^{-1}.
  \end{aligned}
\]
To see this correspondence, simply reflect Figure 14 of \cite{Szep_curv} across the vertical  axis passing through the second crosscap. 

Therefore, we can easily rewrite Relations (S1)--(S13) into our set of generators. In most cases this is straightforward, but note the following remarks.
\begin{itemize}
 \item Generators $A_3,D_1,D_2,D_3$ and $C_2$ are superfluous in the above presentation -- they can be easily removed by use of Relations (S2)\&(S4), (S8), (S10), (S9) and (S12) respectively.
 \item By the above remark, to ensure that $C_1$ and $C_2$ are central in ${\Cal{M}}(N_{3,2})$ (Relation (S13)), it is enough to have commutativity of these elements with $A_1,A_2,B$ and~$U$. 
 \item In order to show that Relations (13a)--(13e) fully replace Relations (S13), we still need to show that $d_2=d_1^{-1}(a_2e_2e_1a_1)^3$ commutes with $a_1,a_2$ and $e_1$, or equivalently, that $(a_2e_2e_1a_1)^3$ commutes with $a_1,a_2$ and $e_1$. 
 \end{itemize}
 Let us first show this in the most difficult of these cases -- the case of the commutativity with $a_2$. The computations repeatedly make use of Relations (1a)--(2c).
\[\begin{aligned}
    a_2^{-1}(a_2e_2e_1a_1)^3a_2&=
    \SR{a_2^{-1}}a_2e_2e_1a_1a_2\SR{e_2}e_1a_1a_2e_2e_1a_1a_2\\
    &=a_2a_2^{-1}e_2e_1a_1a_2\SR{e_1}a_1[e_2a_2e_2]e_1a_1a_2\\
    &=a_2a_2^{-1}e_2e_1[a_1a_2a_1]e_1a_2e_2a_2e_1a_1a_2\\
    &=a_2a_2^{-1}e_2e_1a_2a_1[a_2e_1a_2]e_2a_2e_1a_1a_2\\
    &=a_2a_2^{-1}e_2e_1a_2a_1\SL{e_1}a_2\SR{e_1}e_2a_2e_1a_1a_2\\
    &=a_2a_2^{-1}e_2[e_1a_2e_1]a_1a_2e_2[e_1a_2e_1]a_1a_2\\
    &=a_2[a_2^{-1}e_2a_2]e_1a_2a_1a_2e_2a_2e_1[a_2a_1a_2]\\
    &=a_2e_2a_2\SR{e_2^{-1}}e_1[a_2a_1a_2]e_2a_2\SR{e_1}a_1a_2a_1\\
    &=a_2e_2a_2e_1\SR{e_2^{-1}}a_1a_2a_1\SL{e_2}a_2a_1e_1a_2a_1\\
    &=a_2e_2a_2e_1a_1[e_2^{-1}a_2e_2][a_1a_2a_1]e_1a_2a_1\\
    &=a_2e_2a_2e_1a_1a_2e_2\ukre{a_2}^{-1}\ukre{a_2}\SL{a_1}[a_2e_1a_2]a_1\\
    &=a_2e_2a_2e_1[a_1a_2a_1]e_2e_1a_2e_1a_1\\
    &=a_2e_2[a_2e_1a_2]a_1a_2e_2\SL{e_1}a_2e_1a_1\\
    &=a_2e_2e_1a_2\SR{e_1}a_1a_2e_1e_2a_2e_1a_1\\
    &=a_2e_2e_1a_2a_1[e_1a_2e_1]e_2a_2e_1a_1\\
    &=a_2e_2e_1[a_2a_1a_2]e_1[a_2e_2a_2]e_1a_1\\
    &=(a_2e_2e_1a_1)a_2\SR{a_1}e_1\SL{e_2}(a_2e_2e_1a_1)=(a_2e_2e_1a_1)^3.
   \end{aligned}
\]
Now we show the commutativity of $(a_2e_2e_1a_1)^3$ with $a_1$.
\[\begin{aligned}
    a_1^{-1}(a_2e_2e_1a_1)^3a_1&=
    a_1^{-1}a_2e_2e_1\SL{a_1}a_2e_2e_1a_1a_2e_2e_1\SL{a_1}a_1\\
    &=[a_1^{-1}a_2a_1]\SR{e_2}e_1a_2e_2e_1[a_1a_2a_1]e_2e_1a_1\\
    &=a_2a_1a_2^{-1}e_1[e_2a_2e_2]e_1a_2a_1(a_2e_2e_1a_1)\\
    &=a_2a_1[a_2^{-1}e_1a_2]e_2[a_2e_1a_2]a_1(a_2e_2e_1a_1)\\
    &=a_2a_1e_1a_2\ukre{e_1^{-1}}e_2\ukre{e_1}a_2e_1a_1(a_2e_2e_1a_1)\\
    &=a_2a_1e_1[a_2e_2a_2]e_1a_1(a_2e_2e_1a_1)\\
    &=a_2\SR{a_1}e_1\SL{e_2}(a_2e_2e_1a_1)(a_2e_2e_1a_1)=(a_2e_2e_1a_1)^3.
   \end{aligned}
\]
And finally, we show the commutativity of $(a_2e_2e_1a_1)^3$ with $e_1$.
\[\begin{aligned}
    e_1^{-1}(a_2e_2e_1a_1)^3e_1&=
    e_1^{-1}a_2e_2\SL{e_1}a_1a_2e_2\SR{e_1}a_1a_2e_2e_1a_1\SL{e_1}\\
    &=[e_1^{-1}a_2e_1]\SR{e_2}a_1a_2e_2a_1[e_1a_2e_1]e_2e_1a_1\\
    &=a_2e_1a_2^{-1}a_1[e_2a_2e_2]a_1a_2e_1(a_2e_2e_1a_1)\\
    &=a_2e_1[a_2^{-1}a_1a_2]e_2[a_2a_1a_2]e_1(a_2e_2e_1a_1)\\
    &=a_2e_1a_1a_2\ukre{a_1^{-1}}e_2\ukre{a_1}a_2a_1\SL{e_1}(a_2e_2e_1a_1)\\
    &=a_2e_1a_1[a_2e_2a_2]e_1a_1(a_2e_2e_1a_1)\\
    &=a_2e_1\SR{a_1}\SL{e_2}(a_2e_2e_1a_1)(a_2e_2e_1a_1)=(a_2e_2e_1a_1)^3.
   \end{aligned}
\]
\end{proof}
\begin{uw}
 As we observed in the proof of Theorem \ref{tw:pres:two:holes}, generators: $e_2$, $b_1$, $b_2$, $b_3$, $d_2$ are superfluous in the above presentation. However, we decided to leave these generators in order to have simpler relations. This simplifies the computations in the proofs of Theorems \ref{MainThm1} and \ref{MainThm2}.
\end{uw}
\begin{tw}\label{tw:pres:two:punct}
 The mapping class group ${\Cal{PM}^+}(N_{3}^2)$ of a non--orientable surface of genus 3 with two punctures  admits a presentation with generators $\{a_1,a_2,e_1,e_2,b_1,b_2,b_3, u\}$. The defining relations are
 Relations (1a)--(7),(9),(10) from the statement of Theorem \ref{tw:pres:two:holes} and additionally
 \begin{itemize}
  \item[(8')] $(e_1u)^2=b_1$,
  \item[(11')] $b_3b_1(b_2u)^2=u^2$,
  \item[(12')] $(a_2e_2e_1a_1)^3=1$.
 \end{itemize}
\end{tw}
\begin{proof}
 By Theorem 7.17 of \cite{Szep_curv}, the mapping class group ${\Cal{PM}^+}(N_{3}^2)$ admits a presentations with Relations (S1)--(S7), (S9), (S10) from the proof of Theorem \ref{tw:pres:two:holes} and additionally
 \begin{itemize}
  \item[(S8')] $(UA_2)^2=D_1$,
  \item[(S11')] $(UD_2)^2D_1D_3=U^2$,
  \item[(S12')] $(A_1A_2A_3B)^3=1$.
 \end{itemize}
 If we rewrite Relations (S8'), (S11') and (S12') into our set of generators, we get Relations (8'), (11') and (12') respectively.
\end{proof}
 \section{Action of ${\Cal{M}}(N_{3,s}^n)$ on $H_1(N_{3,s}^n;\zz)$, where $(s,n)=(0,2)$ or $(s,n)=(2,0)$}
In the rest of the paper assume that $(s,n)=(0,2)$ or $(s,n)=(2,0)$.
 Let $\gamma_1,\gamma_2,\gamma_3,\delta_1,\delta_2$ be circles indicated in Figure \ref{r01}. Note that $\gamma_1,\gamma_2,\gamma_3$ are one--sided, $\delta_1,\delta_2$ are two--sided and the $\zz$-module $H_1(N_{3,s}^n;\zz)$ is freely generated by homology classes $[\gamma_1],[\gamma_2],[\gamma_3],[\delta_1]$. In abuse of notation we will not distinguish between the curves $\gamma_1,\gamma_2,\gamma_{3},\delta_1$ and their cycle classes.

The mapping class group ${\Cal{M}}(N_{3,s}^n)$ acts on $H_1(N_{3,s}^n;\zz)$, hence we have a representation
\[\map{\psi}{{\Cal{M}}(N_{3,s}^n)}{\textrm{Aut}(H_1(N_{3,s}^n;\zz))}. \]
It is straightforward to check that
\begin{equation}\begin{aligned}\label{eq:psi:1}
   \psi(a_1)&=\begin{bmatrix}
0&1&0&0\\
-1&2&0&0\\
0&0&1&0\\
0&0&0&1\end{bmatrix},\psi(a_1^{-1})=\begin{bmatrix}
2&-1&0&0\\
1&0&0&0\\
0&0&1&0\\
0&0&0&1
    \end{bmatrix} \\
\psi(a_2)&=\begin{bmatrix}
1&0&0&0\\
0&0&1&0\\
0&-1&2&0\\
0&0&0&1\end{bmatrix}, \psi(a_2^{-1})=\begin{bmatrix}
1&0&0&0\\
0&2&-1&0\\
0&1&0&0\\
0&0&0&1
    \end{bmatrix}\\
    \psi(u)&=\psi(u^{-1})=\begin{bmatrix}
    0&1&0&0\\
1&0&0&0\\
0&0&1&0\\
0&0&0&1
    \end{bmatrix}
\end{aligned} \end{equation}
    \begin{equation}\begin{aligned} \label{eq:psi:2}
    \psi(e_1)&=
    \begin{bmatrix}0&1&0&0\\
     -1&2&0&0\\
     0&0&1&0\\
     -1&1&0&1
    \end{bmatrix}, \psi(e_1^{-1})=
    \begin{bmatrix}2&-1&0&0\\
     1&0&0&0\\
     0&0&1&0\\
     1&-1&0&1
    \end{bmatrix}\\
    \psi(e_2)&=
    \begin{bmatrix}2&-1&0&0\\
     1&0&0&0\\
     2&-2&1&0\\
     0&0&0&1
    \end{bmatrix}, \psi(e_2^{-1})=
    \begin{bmatrix}0&1&0&0\\
     -1&2&0&0\\
     -2&2&1&0\\
     0&0&0&1
    \end{bmatrix}\\
    \psi(b_j^{\pm 1})&=\psi(d_j^{\pm 1})=
    I_4\\
    \end{aligned}\end{equation}    
where $I_4$ is the identity matrix of rank 4.
\section{Computing $\gen{\kre{X}}\cap \ker\kre{\partial}_1$}
Let $G={\Cal{M}}(N_{3,s}^n)$ and $M=H_1(N_{3,s}^n;\zz)$  as in the previous section, and define
\[\xi_i=\begin{cases}
         \gamma_i&\text{for $i=1,2,3$,}\\
         \delta_{1}&\text{for $i=4$.}
        \end{cases}
\]
If $h\in G$, then
\[\kre{\partial}_1([h]\otimes\xi_i)=(h-1)[\cdot]\otimes\xi_i=(\psi(h)^{-1}-I_4)\xi_i,\]
where we identified $C_0(G)\otimes M$ with $M$ by the map $[\cdot]\otimes m\mapsto m$.

Let us denote 
\[[a_j]\otimes \xi_i,\ [u]\otimes\xi_i,\ [e_j]\otimes\xi_i,\ [b_j]\otimes\xi_i,\
[d_j]\otimes\xi_i\] 
respectively by 
\[a_{j,i},\ u_{i},\ e_{j,i},\ b_{j,i},\ d_{j,i}.\]
Using the formulas \eqref{eq:psi:1} and \eqref{eq:psi:2}, we obtain
\begin{equation*}\begin{aligned}
\kre{\partial}_1(a_{j,i})&=\begin{cases}
                           \gamma_j+\gamma_{j+1}&\text{if $i=j$}\\
                           -\gamma_j-\gamma_{j+1}&\text{if $i=j+1$}\\
                           0&\text{otherwise,}
                          \end{cases}
\\
\kre{\partial}_1(u_{i})&=\begin{cases}
                           -\gamma_1+\gamma_{2}&\text{if $i=1$}\\
                           \gamma_1-\gamma_{2}&\text{if $i=2$}\\
                           0&\text{otherwise,}
                          \end{cases}\\
\kre{\partial}_1(e_{1,i})&=\begin{cases}
                           \gamma_1+\gamma_2+\delta_1&\text{if $i=1$}\\
                           -\gamma_1-\gamma_2-\delta_1&\text{if $i=2$}\\
                           0&\text{otherwise,}
                          \end{cases}
\\
\kre{\partial}_1(e_{2,i})&=\begin{cases}
                           -\gamma_1-\gamma_2-2\gamma_3&\text{if $i=1$}\\
                           \gamma_1+\gamma_2+2\gamma_3&\text{if $i=2$}\\
                           0&\text{otherwise,}
                          \end{cases}\\
\kre{\partial}_1(b_{j,i})&=\kre{\partial}_1(d_{j,i})=0.
\end{aligned}\end{equation*}
The above formulas show that all of the following elements are contained in $\ker\kre{\partial}_1$
 \begin{enumerate}
 \item[(K1)] $a_{j,i}$ for $j=1,2$ and $i\in\{1,2,3,4\}\bez\{j,j+1\}$,
\item[(K2)] $a_{j,j}+a_{j,j+1}$ for $j=1,2$,
\item[(K3)] $u_{i}$ for $i=3,4$,
\item[(K4)] $u_{1}+u_{2}$,
\item[(K5)] $e_{j,i}$ for $j=1,2$ and $i=3,4$,
\item[(K6)] $e_{j,1}+e_{j,2}$ for $j=1,2$,
\item[(K7)] $e_{2,1}+2a_{2,2}-u_{1}$,
\item[(K8)] $b_{j,i}$ for $j=1,2,3$ and $i=1,2,3,4$,
\item[(K9)] $d_{j,i}$ for $j=1,2$ and $i=1,2,3,4$.
\end{enumerate}
\begin{prop}\label{prop:kernel:1}
 Let $G={\Cal{M}}(N_{3,2})$. Then $\gen{\kre{X}}\cap \ker\kre{\partial}_1$ is the abelian group generated freely by Generators (K1)--(K9).
\end{prop}
\begin{proof}
 By Theorem \ref{tw:pres:two:holes}, $\gen{\kre{X}}$ is generated freely by $\{a_{j,i},u_{i},e_{j,i},b_{j,i},d_{j,i}\}$.
Suppose that $h\in\gen{\kre{X}}\cap\ker\kre{\partial}_1$. We will show that $h$ can be uniquely expressed as a linear combination of Generators (K1)--(K9) specified in the statement of the theorem.

We decompose $h$ as follows:
\begin{itemize}
 \item $h=h_0=h_1+h_2$, where $h_1$ is a combination of Generators (K1)--(K2) and $h_2$ does not contain $a_{j,i}$ with $i\neq j$;
 \item $h_2=h_3+h_4$, where $h_3$ is a combination of Generators (K3)--(K4) and $h_4$ does not contain $u_i$ with $i\neq 1$;
 \item $h_4=h_5+h_6$, where $h_5$ is a combination of Generators (K5)--(K7) and $h_6$ does not contain $e_{j,i}$ for $(i,j)\neq (1,1)$;
 \item $h_6=h_7+h_8$, where $h_7$ is a combination of Generators (K8) and $h_8$ does not contain $b_{j,i}$;
 \item $h_8=h_9+h_{10}$, where $h_9$ is a combination of Generators (K9) and $h_{10}$ does not contain $d_{j,i}$. 
\end{itemize}
Observe also that for each $k=0,\ldots,8$, $h_{k+1}$ and $h_{k+2}$ are uniquely determined by $h_k$. Element $h_{10}$ has the form 
\[h_{10}=k_1a_{1,1}+k_2a_{2,2}+l e_{1,1}+m u_{1}\]
for some integers $k_1,k_2,l,m$. Hence
\[\begin{aligned}
 0=&\kre{\partial}_1(h_{10})=k_1(\gamma_1+\gamma_2)+k_2(\gamma_2+\gamma_3)+l(\gamma_1+\gamma_2+\delta_1)+m(-\gamma_1+\gamma_2).
\end{aligned}\]
This implies that $l=0, k_2=0$, and then $k_1=m=0$ and thus $h_{10}=0$.
\end{proof}
By an analogous argument, Theorem \ref{tw:pres:two:punct} implies that
\begin{prop}\label{prop:kernel:2}
 Let $G={\Cal{PM}}^+(N_{3}^2)$. Then $\gen{\kre{X}}\cap \ker\kre{\partial}_1$ is the abelian group generated freely by Generators (K1)--(K8).
\end{prop}
 \section{Computing $H_1({\Cal{M}}(N_{3,2});H_1(N_{3,2};\zz))$}
 The goal of this section is to prove that
\begin{prop}\label{prop:h1:1}  The abelian group $H_1({\Cal{PM}^+}(N_{3}^2);H_1(N_{3}^2;\zz))$  has a presentation with generators
\[a_{1,3},\ a_{1,4},\ a_{1,1}+a_{1,2},\ u_3,\ u_4,\ d_{1,1},\]
and relations
\[2a_{1,3}=2a_{1,4}=2\left(a_{1,1}+a_{1,2}\right)=2u_3=2u_4=2d_{1,1}=0.\]
\end{prop}
\begin{proof} 
We start with the generators provided by Proposition \ref{prop:kernel:1}, and using the formula \eqref{eq_rew_rel}, we rewrite relations from Theorem \ref{tw:pres:two:holes} as relations in 
$H_1({\Cal{M}}(N_{3,2});H_1(N_{3,2};\zz))$.

\subsection*{(3)}
Relation (3) gives
\[\begin{aligned}
   0&=([a_1]+a_1[u]+a_1u[a_1]-[u])\otimes \xi_i\\
   &=[a_1]\otimes (I_4+\psi(u^{-1}a_1^{-1}))\xi_i+[u]\otimes(\psi(a_1^{-1})-I_4)\xi_i\\
   &=[a_1]\otimes \begin{bmatrix}
                   2&0&0&0\\
                   2&0&0&0\\
                   0&0&2&0\\
                   0&0&0&2
                  \end{bmatrix}\xi_i+[u]\otimes 
    \begin{bmatrix}
                   1&-1&0&0\\
                   1&-1&0&0\\
                   0&0&0&0\\
                   0&0&0&0
                  \end{bmatrix}\xi_i.
  \end{aligned}
\]
We conclude from this relation that Generator (K4) is trivial:
\begin{equation}u_1+u_2=0\label{k4:0},\end{equation}
and each of the generators $a_{1,3},\ a_{1,4},\ a_{1,1}+a_{1,2}$ is of order at most 2: 
\begin{equation}\label{a1:mod2}
 2a_{1,3}=2a_{1,4}=2(a_{1,1}+a_{1,2})=0. 
\end{equation}
\subsection*{(2a)--(2c)}
Relation (2a) 
gives 
\[\begin{aligned}0=&
([a_1]+a_{1}[a_{2}]+a_{1}a_{2}[a_{1}]-[a_{2}]-a_{2}[a_{1}]-a_{2}a_{1}[a_{2}])\otimes \xi_i\\
=&[a_1]\otimes (I_4+\psi(a_{2}^{-1}a_1^{-1})-\psi(a_{2}^{-1}))\xi_i\\
&+[a_{2}]\otimes(\psi(a_1^{-1})-I_4-\psi(a_1^{-1}a_{2}^{-1}))\xi_i\\
=& [a_1]\otimes \begin{bmatrix}
                   2&-1&0&0\\
                   2&-1&0&0\\
                   1&-1&1&0\\
                   0&0&0&1
                  \end{bmatrix}\xi_i+[a_2]\otimes \begin{bmatrix}
                   -1&1&-1&0\\
                   0&-1&0&0\\
                   0&-1&0&0\\
                   0&0&0&-1
                  \end{bmatrix}\xi_i.
\end{aligned}\]
We conclude from this relation, that
\begin{equation} \label{a24:a14}
 \begin{aligned}
   a_{2,4}&=a_{1,4},\\
   a_{2,1}&=a_{1,3}.
  \end{aligned}
\end{equation}
Hence generators $a_{2,1}$ and $a_{2,4}$ are superfluous. Moreover,
\begin{equation}\label{a22:a23}
 a_{2,2}+a_{2,3}=a_{1,1}+a_{1,2},
\end{equation}
which together with the formula \eqref{a1:mod2} imply that Generators (K2) generate a cyclic group of order at most 2.

Relation (2b) gives
\[\begin{aligned}0=&
([e_1]+e_{1}[a_{2}]+e_{1}a_{2}[e_{1}]-[a_{2}]-a_{2}[e_{1}]-a_{2}e_{1}[a_{2}])\otimes \xi_i\\
=&[e_1]\otimes (I_4+\psi(a_{2}^{-1}e_1^{-1})-\psi(a_{2}^{-1}))\xi_i\\
&+[a_{2}]\otimes(\psi(e_1^{-1})-I_4-\psi(e_1^{-1}a_{2}^{-1}))\xi_i\\
=&[e_1]\otimes\begin{bmatrix}
                   2&-1&0&0\\
                   2&-1&0&0\\
                   1&-1&1&0\\
                   1&-1&0&1
                  \end{bmatrix} \xi_i +[a_2]\otimes\begin{bmatrix}
                   -1&1&-1&0\\
                   0&-1&0&0\\
                   0&-1&0&0\\
                   0&1&-1&-1
                  \end{bmatrix} \xi_i.
\end{aligned}\]
This, the formulas \eqref{a1:mod2}, \eqref{a24:a14} and \eqref{a22:a23} imply that 
\begin{equation} \label{e1:s}
\begin{aligned}
   e_{1,4}&=a_{2,4}=a_{1,4},\\
   e_{1,3}&=a_{2,1}+a_{2,4}=a_{1,3}+a_{1,4},\\
   e_{1,1}+e_{1,2}&=(a_{2,2}+a_{2,3})+a_{2,4}=(a_{1,1}+a_{1,2})+a_{1,4}
  \end{aligned} 
\end{equation}
and Generators (K5) and (K6) are superfluous for $j=1$.

Relation (2c) gives 
\[\begin{aligned}0=&
([e_2]+e_{2}[a_{2}]+e_{2}a_{2}[e_{2}]-[a_{2}]-a_{2}[e_{2}]-a_{2}e_{2}[a_{2}])\otimes \xi_i\\
=&[e_2]\otimes (I_4+\psi(a_{2}^{-1}e_2^{-1})-\psi(a_{2}^{-1}))\xi_i\\
&+[a_{2}]\otimes(\psi(e_2^{-1})-I_4-\psi(e_2^{-1}a_{2}^{-1}))\xi_i\\
=&[e_2]\otimes \begin{bmatrix}
                0&1&0&0\\
                0&1&0&0\\
                -1&1&1&0\\
                0&0&0&1
               \end{bmatrix}\xi_i+[a_2]\otimes \begin{bmatrix}
                -1&-1&1&0\\
                0&-3&2&0\\
                0&-3&2&0\\
                0&0&0&-1
               \end{bmatrix}\xi_i.
\end{aligned}\]
This relation, the formulas \eqref{a1:mod2}, \eqref{a24:a14} and \eqref{a22:a23} imply that
\begin{equation}\label{e2:s}
 \begin{aligned}
   e_{2,4}&=a_{2,4}=a_{1,4},\\
   e_{2,3}&=-a_{2,1}=a_{1,3},\\
   e_{2,1}+e_{2,2}&=a_{2,2}+a_{2,3}=a_{1,1}+a_{1,2}.
  \end{aligned}
\end{equation}
Thus Generators (K5) and (K6) are superfluous for $j=2$.
\subsection*{(1a)--(1c)}
Relation (1a) gives
\[\begin{aligned}
   0&=([a_1]+a_1[e_1]-[e_1]-e_1[a_1])\otimes \xi_i\\
   &=[a_1]\otimes (I_4-\psi(e_1^{-1}))\xi_i+[e_1]\otimes(\psi(a_1^{-1})-I_4)\xi_i\\
   &=[a_1]\otimes \begin{bmatrix}
                   -1&1&0&0\\
                   -1&1&0&0\\
                   0&0&0&0\\
                   -1&1&0&0
                  \end{bmatrix}\xi_i+[e_1]\otimes \begin{bmatrix}
                  1&-1&0&0\\
                  1&-1&0&0\\
                  0&0&0&0\\
                  0&0&0&0
                  \end{bmatrix}\xi_i.
  \end{aligned}
\]
This relation gives no new information.

Relation (1b) gives
\[\begin{aligned}
   0&=([a_1]+a_1[e_2]-[e_2]-e_2[a_1])\otimes \xi_i\\
   &=[a_1]\otimes (I_4-\psi(e_2^{-1}))\xi_i+[e_2]\otimes(\psi(a_1^{-1})-I_4)\xi_i\\
   &=[a_1]\otimes \begin{bmatrix}
                   1&-1&0&0\\
                   1&-1&0&0\\
                   2&-2&0&0\\
                   0&0&0&0
                  \end{bmatrix}\xi_i+[e_2]\otimes \begin{bmatrix}
                  1&-1&0&0\\
                  1&-1&0&0\\
                  0&0&0&0\\
                  0&0&0&0
                  \end{bmatrix}\xi_i.
  \end{aligned}
\]
This relation gives no new information.

Relation (1c) gives 
\[\begin{aligned}
   0&=([e_1]+e_1[e_2]-[e_2]-e_2[e_1])\otimes \xi_i\\
   &=[e_1]\otimes (I_4-\psi(e_2^{-1}))\xi_i+[e_2]\otimes(\psi(e_1^{-1})-I_4)\xi_i\\
   &=[e_1]\otimes \begin{bmatrix}
                   1&-1&0&0\\
                   1&-1&0&0\\
                   2&-2&0&0\\
                   0&0&0&0
                  \end{bmatrix}\xi_i+[e_2]\otimes \begin{bmatrix}
                  1&-1&0&0\\
                  1&-1&0&0\\
                  0&0&0&0\\
                  1&-1&0&0
                  \end{bmatrix}\xi_i.
  \end{aligned}
\]
This relation gives no new information.
\subsection*{(4)}
Relation (4) gives
\[\begin{aligned}
   0=&([u]+u[e_2]-[a_2]-a_2[u]-a_2u[e_2]-a_2ue_2[a_2])\otimes \xi_i\\
   =&[u]\otimes (I_4-\psi(a_2^{-1}))\xi_i+[e_2]\otimes(\psi(u^{-1})-\psi(u^{-1}a_2^{-1}))\xi_i\\
   &+[a_2]\otimes(-I_4-\psi(e_2^{-1}u^{-1}a_2^{-1}))\xi_i\\
   =&[u]\otimes \begin{bmatrix}
                   0&0&0&0\\
                   0&-1&1&0\\
                   0&-1&1&0\\
                   0&0&0&0
                  \end{bmatrix}\xi_i+[e_2]\otimes \begin{bmatrix}
                  0&-1&1&0\\
                  0&0&0&0\\
                  0&-1&1&0\\
                  0&0&0&0
                  \end{bmatrix}\xi_i\\
                  &+[a_2]\otimes \begin{bmatrix}
                  -2&0&0&0\\
                  -2&1&-1&0\\
                  -2&3&-3&0\\
                  0&0&0&-2
                  \end{bmatrix}\xi_i.
  \end{aligned}
\]
This relation, the formulas \eqref{k4:0}, \eqref{a1:mod2}, \eqref{a22:a23} and \eqref{e2:s} give
\begin{equation}
\begin{aligned}
 e_{2,1}+2a_{2,2}-u_1&=3(a_{2,2}+a_{2,3})-(u_1+u_2)-u_3-e_{2,3}\\
 &=(a_{1,1}+a_{1,2})-u_3+a_{1,3}
\end{aligned}\label{rel_e4}
\end{equation}
which implies that Generator (K7) is superfluous.
\subsection*{(9b)}
If we denote 
\[\begin{aligned}
   M&=I+\psi(u^{-1}e_2^{-1})=\begin{bmatrix}
                                0&2&0&0\\
                                0&2&0&0\\
                                -2&2&2&0\\
                                0&0&0&2
                             \end{bmatrix} \\
   N&=I+\psi(a_1^{-2}e_2^{-1}a_2^{-1})+\psi(a_1^{-2}e_2^{-1}a_2^{-1})^2=\begin{bmatrix}
                                 3&-1&1&0\\
                                 3&-1&1&0\\
                                 3&-1&1&0\\
                                 0&0&0&3
                                     \end{bmatrix}
  \end{aligned}
\]
then Relation (9b) gives
\[\begin{aligned}
   0=&[e_2]\otimes \left(M-\psi(a_2^{-1})N\right)\xi_i+[u]\otimes \psi(e_2^{-1})M\xi_i-[a_2]\otimes N\xi_i\\
   &-[a_1]\otimes\left(\psi(e_2^{-1}a_2^{-1})N+\psi(a_1^{-1}e_2^{-1}a_2^{-1})N\right)\xi_i\\
   =&[e_2]\otimes \begin{bmatrix}
                   -3&3&-1&0\\
                   -3&3&-1&0\\
                   -5&3&1&0\\
                   0&0&0&-1
                  \end{bmatrix}\xi_i+[u]\otimes \begin{bmatrix}
                  0&2&0&0\\
                  0&2&0&0\\
                  -2&2&2&0\\
                  0&0&0&2
                  \end{bmatrix}\xi_i\\
                  &+[a_2]\otimes \begin{bmatrix}
                  -3&1&-1&0\\
                  -3&1&-1&0\\
                  -3&1&-1&0\\
                  0&0&0&-3
                  \end{bmatrix}\xi_i+[a_1]\otimes \begin{bmatrix}
                  -6&2&-2&0\\
                  -6&2&-2&0\\
                  -6&2&-2&0\\
                  0&0&0&-6
                  \end{bmatrix}\xi_i.
  \end{aligned}
\]
The last two columns of this relation, the formulas \eqref{a1:mod2}, \eqref{a24:a14}, \eqref{a22:a23} and \eqref{e2:s}  imply that 
\begin{equation}\label{u34:mod2}
 \begin{aligned}
  2u_{3}=&(e_{2,1}+e_{2,2}+a_{2,2}+a_{2,3})+(a_{2,1}-e_{2,3})+\\
  &+2(a_{1,1}+a_{1,2})+2a_{1,3}=0,\\
  2u_{4}=&e_{2,4}+3a_{2,4}+6a_{1,4}=0.
  \end{aligned}
\end{equation}
As for the first two columns of this relation 
\begin{equation*}
\begin{aligned}
 -2u_3=&3(e_{2,1}+e_{2,2}+a_{2,2}+a_{2,3})+5(a_{2,1}+e_{2,3})+\\
  &+6(a_{1,1}+a_{1,2})+6a_{1,3}-2a_{2,1}=0,\\
  -2u_{3}=&(3e_{2,1}+3e_{2,2}+2a_{1,1}+2a_{1,2}+a_{2,2}+a_{2,3})+3(a_{2,1}+e_{2,3})+\\
  &+2(u_{1,1}+u_{1,2})+2a_{1,3}-2a_{2,1}=0,\\
\end{aligned}
\end{equation*}
they do not give any additional information.
\subsection*{(9a)}
Let $M=I+\psi(u^{-1}e_2^{-1})$ as above. Then Relation (9a) gives
\[\begin{aligned}
   0&=[e_2]\otimes M\xi_i+[u]\otimes(\psi(e_2^{-1})M)\xi_i-[b_3]\otimes \xi_i=\\
   &=[e_2]\otimes \begin{bmatrix}
                                0&2&0&0\\
                                0&2&0&0\\
                                -2&2&2&0\\
                                0&0&0&2
                             \end{bmatrix}\xi_i+[u]\otimes 
                             \begin{bmatrix}
                                0&2&0&0\\
                                0&2&0&0\\
                                -2&2&2&0\\
                                0&0&0&2
                             \end{bmatrix}\xi_i-[b_3]\otimes \xi_i.
  \end{aligned}
\]
This, the formulas \eqref{k4:0}, \eqref{a1:mod2},  \eqref{e2:s} and \eqref{u34:mod2} imply that generators
\begin{equation}\label{b3:s}
 \begin{aligned}
   b_{3,4}&=2e_{2,4}+2u_4=0,\\
   b_{3,3}&=2e_{2,3}+2u_3=0,\\
   b_{3,1}&=-2e_{2,3}-2u_{3}=0,\\
   b_{3,2}&=2(e_{2,1}+e_{2,2})+2e_{2,3}+2(u_1+u_2)+2u_3=0
  \end{aligned}
\end{equation}
are trivial.
\subsection*{(12)}
Let 
\[M=I+\psi(a_1^{-1}e_1^{-1}e_2^{-1}a_2^{-1})+\psi(a_1^{-1}e_1^{-1}e_2^{-1}a_2^{-1})^2=\begin{bmatrix}
                                3&-1&1&0\\
                                3&-1&1&0\\
                                3&-1&1&0\\
                                0&-1&1&3
                             \end{bmatrix}.\]
Relation (12), the formulas \eqref{a1:mod2}, \eqref{a24:a14}, \eqref{a22:a23}, \eqref{e1:s} and \eqref{e2:s} give
\[\begin{aligned}
   0=&[a_2]\otimes M\xi_i+[e_2]\otimes(\psi(a_2^{-1})M)\xi_i+[e_1]\otimes(\psi(e_2^{-1}a_2^{-1})M)\xi_i\\
   &+[a_1]\otimes(\psi(e_1^{-1}e_2^{-1}a_2^{-1})M)\xi_i-[d_1]\otimes\xi_i-[d_2]\otimes\psi(d_1^{-1})\xi_i\\
   =&([a_2]+[e_2]+[e_1]+[a_1])\otimes\begin{bmatrix}
                                3&-1&1&0\\
                                3&-1&1&0\\
                                3&-1&1&0\\
                                0&-1&1&3
                             \end{bmatrix}\xi_i-[d_1]\otimes\xi_i-[d_2]\otimes\xi_i=\\
                             =&-[d_1]\otimes\xi_i-[d_2]\otimes\xi_i.
  \end{aligned}
\]
This implies that 
\begin{equation}\label{d2:s}
 d_{2,i}=-d_{1,i},\quad\text{for $i=1,2,3,4$.}
\end{equation}
\subsection*{(13a)--(13e)}
Relation (13a) gives
\[\begin{aligned}
   0&=([d_1]+d_1[a_1]-[a_1]-a_1[d_1])\otimes \xi_i\\
   &=[d_1]\otimes (I_4-\psi(a_1^{-1}))\xi_i+[a_1]\otimes(\psi(d_1^{-1})-I_4)\xi_i\\
   &=[d_1]\otimes \begin{bmatrix}
                  -1&1&0&0\\
                  -1&1&0&0\\
                  0&0&0&0\\
                  0&0&0&0
                  \end{bmatrix}\xi_i.
  \end{aligned}
\]
Thus 
\[d_{1,2}=-d_{1,1}.\]
By a similar argument, Relation (13d) gives 
\[\begin{aligned}
   0&=[d_1]\otimes \begin{bmatrix}
                  1&-1&0&0\\
                  -1&1&0&0\\
                  0&0&0&0\\
                  0&0&0&0
                  \end{bmatrix}\xi_i.
  \end{aligned}
\]
Hence 
\begin{equation}\label{d12:s}
d_{1,2}=d_{1,1},
 \end{equation}
and
\begin{equation}\label{d11:mod2}
2d_{1,1}=0.
\end{equation}
Analogously, Relation (13b) gives
\[\begin{aligned}
   0&=[d_1]\otimes \begin{bmatrix}
                  0&0&0&0\\
                  0&-1&1&0\\
                  0&-1&1&0\\
                  0&0&0&0
                  \end{bmatrix}\xi_i,
  \end{aligned}
\]
which implies that 
\[d_{1,3}=-d_{1,2}.\]
This together with the formulas \eqref{d12:s} and \eqref{d11:mod2} imply that
\begin{equation}\label{d13:s}
 d_{1,3}=d_{1,1}.
\end{equation}
Relation (13c) gives 
\[\begin{aligned}
   0&=([d_1]+d_1[e_1]-[e_1]-e_1[d_1])\otimes \xi_i\\
   &=[d_1]\otimes (I_4-\psi(e_1^{-1}))\xi_i+[e_1]\otimes(\psi(d_1^{-1})-I_4)\xi_i\\
   &=[d_1]\otimes \begin{bmatrix}
                  -1&1&0&0\\
                  -1&1&0&0\\
                  0&0&0&0\\
                  -1&1&0&0
                  \end{bmatrix}\xi_i.
  \end{aligned}
\]
Hence, by the formulas \eqref{d12:s} and \eqref{d11:mod2}
\begin{equation}\label{d14:s}
 d_{1,4}=-d_{1,1}-d_{1,2}=0
\end{equation}
and we conclude that Generators (K9) generate a cyclic group of order at most 2.

Relation (13e)
\[\begin{aligned}
   0&=[d_2]\otimes \begin{bmatrix}
                  1&-1&0&0\\
                  -1&1&0&0\\
                  0&0&0&0\\
                  0&0&0&0
                  \end{bmatrix}\xi_i
  \end{aligned}
\]
gives no new information.
\subsection*{(8)}
Let 
\[M=I+\psi(u^{-1}e_1^{-1})=\begin{bmatrix}
                                2&0&0&0\\
                                2&0&0&0\\
                                0&0&2&0\\
                                1&-1&0&2
                             \end{bmatrix}.\]
Relation (8) gives
\[\begin{aligned}
   0&=[e_1]\otimes M\xi_i+[u]\otimes(\psi^{-1}(e_1^{-1})M)\xi_i-[d_1]\otimes \xi_i-[b_1]\otimes \xi_i\\
   &=[e_1]\otimes \begin{bmatrix}
                                2&0&0&0\\
                                2&0&0&0\\
                                0&0&2&0\\
                                1&-1&0&2
                             \end{bmatrix}\xi_i+[u]\otimes 
                             \begin{bmatrix}
                                2&0&0&0\\
                                2&0&0&0\\
                                0&0&2&0\\
                                1&-1&0&2
                             \end{bmatrix}\xi_i-[d_1]\otimes \xi_i-[b_1]\otimes \xi_i.
  \end{aligned}
\]
This, the formulas \eqref{k4:0}, \eqref{a1:mod2}, \eqref{a22:a23}, \eqref{e1:s}, \eqref{u34:mod2}, \eqref{d13:s} and \eqref{d14:s} imply that generators
\begin{equation}\label{b1:s}
 \begin{aligned}
    b_{1,1}&=e_{1,4}+u_4-d_{1,1}=a_{1,4}+u_4+d_{1,1},\\
    b_{1,2}&=-e_{1,4}-u_4-d_{1,2}=a_{1,4}+u_4+d_{1,1},\\
    b_{1,3}&=2e_{1,3}+2u_3-d_{1,3}=d_{1,1},\\
    b_{1,4}&=2e_{1,4}+2u_4-d_{1,4}=0
  \end{aligned}
\end{equation}
are superfluous.
\subsection*{(5), (6)}
Relations (5) and (6) give
\[\begin{aligned}
   0&=([u]+u[b_k]-[b_k]-b_k[u])\otimes \xi_i\\
   &=[u]\otimes (I_4-\psi(b_k^{-1}))\xi_i+[b_k]\otimes(\psi(u^{-1})-I_4)\xi_i\\
   &=[b_k]\otimes \begin{bmatrix}
                  -1&1&0&0\\
                  1&-1&0&0\\
                  0&0&0&0\\
                  0&0&0&0
                  \end{bmatrix}\xi_i
  \end{aligned}
\]
for $k=1,3$. This relation gives no new information.
\subsection*{(10)}
Observe first that if $g\in G$, then the formulas \eqref{diff:formula}, \eqref{eq_rew_rel} and the relation $g^{-1}g=1$  imply that 
\[[g^{-1}]\otimes\xi_i=-g^{-1}[g]\otimes \xi_i.\]
Hence Relation (10) gives 
\[\begin{aligned}
   0=&[u]\otimes\left(I_4+\psi(b_2^{-1}ue_1^{-1}u^{-1})-\psi(a_2b_1a_2^{-1})\right)\xi_i+[u^{-1}]\otimes\psi(e_1^{-1}u^{-1})\xi_i\\
   &+[e_1]\otimes\left(\psi(u^{-1})-\psi(u^{-1}a_2b_1a_2^{-1})\right)\xi_i\\
   &+[b_2]\otimes\psi(ue_1^{-1}u^{-1})\xi_i-
   [b_1^{-1}]\otimes \psi(a_2^{-1})\xi_i\\
   &-[a_2]\otimes \xi_i-[a_2^{-1}]\otimes\psi(b_1a_2^{-1})\xi_i\\
   =&[u]\otimes\left(I_4+\psi(b_2^{-1}ue_1^{-1}u^{-1})-\psi(a_2b_1a_2^{-1})-\psi(ue_1^{-1}u^{-1})\right)\xi_i\\
   &+[e_1]\otimes\left(\psi(u^{-1})-\psi(u^{-1}a_2b_1a_2^{-1})\right)\xi_i\\
   &+[b_2]\otimes\psi(ue_1^{-1}u^{-1})\xi_i+
   [b_1]\otimes \psi(b_1a_2^{-1})\xi_i\\
   &+[a_2]\otimes\left(\psi(a_2b_1a_2^{-1})-I_4\right) \xi_i.
  \end{aligned}
\]
Recall that $\psi(b_j^{\pm 1})=I_4$, hence the above formula takes form   
\[\begin{aligned}
0=&[b_2]\otimes\psi(ue_1^{-1}u^{-1})\xi_i+
   [b_1]\otimes \psi(a_2^{-1})\xi_i\\
   =&[b_2]\otimes \begin{bmatrix}
                                0&1&0&0\\
                                -1&2&0&0\\
                                0&0&1&0\\
                                -1&1&0&1
                             \end{bmatrix}\xi_i+[b_1]\otimes 
                             \begin{bmatrix}
                                1&0&0&0\\
                                0&2&-1&0\\
                                0&1&0&0\\
                                0&0&0&1
                             \end{bmatrix}\xi_i.
  \end{aligned}
\]

This, the formulas \eqref{a1:mod2}, \eqref{u34:mod2} and \eqref{b1:s} imply that generators
\begin{equation}\label{b2:s}
\begin{aligned}
   b_{2,4}&=-b_{1,4}=0,\\
   b_{2,3}&=b_{1,2}=a_{1,4}+u_4+d_{1,1},\\
   b_{2,2}&=b_{1,1}-b_{2,4}=a_{1,4}+u_4+d_{1,1},\\
   b_{2,1}&=-2b_{1,2}-b_{1,3}-2b_{2,2}-b_{2,4}=-d_{1,1}
  \end{aligned} 
\end{equation}
are superfluous.
\subsection*{(7)}
Relation (7) gives 
\[\begin{aligned}
   0&=([a_2]+a_2[b_2]-[b_2]-b_2[a_2])\otimes \xi_i\\
   &=[a_2]\otimes (I_4-\psi(b_2^{-1}))\xi_i+[b_2]\otimes(\psi(a_2^{-1})-I_4)\xi_i\\
   &=[b_2]\otimes \begin{bmatrix}
                  0&0&0&0\\
                  0&1&-1&0\\
                  0&1&-1&0\\
                  0&0&0&0
                  \end{bmatrix}\xi_i.
  \end{aligned}
\]
This relation gives no new information.
\subsection*{(11)}
Relation (11) gives
\[\begin{aligned}
   0=&[b_3]\otimes\xi_i+[b_1]\otimes (\psi(b_3^{-1}))\xi_i+[b_2]\otimes (\psi(b_1^{-1}b_3^{-1})+\psi(u^{-1}b_2^{-1}b_1^{-1}b_3^{-1}))\xi_i\\
   &+[u]\otimes(\psi(b_2^{-1}b_1^{-1}b_3^{-1})+ \psi(b_2^{-1}u^{-1}b_2^{-1}b_1^{-1}b_3^{-1}))-\psi(d_1^{-1}d_2^{-2})
   -\psi(u^{-1}d_1^{-1}d_2^{-2}))\xi_i\\
   &-[d_2]\otimes(I_4+\psi(d_2^{-1}))\xi_i-[d_1]\otimes(\psi(d_2^{-2}))\xi_i\\
   =&[b_3]\otimes\xi_i+[b_1]\otimes \xi_i+[b_2]\otimes (I_4+\psi(u^{-1}))\xi_i-2[d_2]\otimes\xi_i-[d_1]\otimes\xi_i\\
   =&[b_3]\otimes\xi_i+[b_1]\otimes \xi_i+[b_2]\otimes\begin{bmatrix}
                  1&1&0&0\\
                  1&1&0&0\\
                  0&0&2&0\\
                  0&0&0&2
                  \end{bmatrix}\xi_i-2[d_2]\otimes\xi_i-[d_1]\otimes\xi_i.
  \end{aligned}
\]
This relation gives no new information. 

In order to sum up the above computation, recall that $H_1({\Cal{M}}(N_{3,2});H_1(N_{3,2};\zz))$ is the quotient of the free abelian group with basis specified in the statement of Proposition \ref{prop:kernel:1} by the relations obtained above. If we remove the superfluous generators
\begin{itemize}
 \item $u_1+u_2$ by the formula \eqref{k4:0},
 \item $a_{2,1}$ and $a_{2,4}$ by the formula \eqref{a24:a14},
 \item $a_{2,2}+a_{2,3}$ by the formula \eqref{a22:a23},
 \item $e_{1,1}+e_{1,2}$, $e_{1,3}$ and $e_{1,4}$ by the formula \eqref{e1:s},
 \item $e_{2,1}+e_{2,2}$, $e_{2,3}$ and $e_{2,4}$ by the formula \eqref{e2:s},
 \item $e_{2,1}+2a_{2,2}-u_1$ by the formula \eqref{rel_e4}
 \item $d_{1,2}$, $d_{1,3}$, $d_{1,4}$, $d_{2,1}$, $d_{2,2}$, $d_{2,3}$ and $d_{2,4}$ by the formulas \eqref{d12:s}, \eqref{d13:s}, \eqref{d14:s} and \eqref{d2:s},
 \item $b_{1,1}$, $b_{1,2}$, $b_{1,3}$ and $b_{1,4}$ by the formula \eqref{b1:s},
 \item $b_{2,1}$, $b_{2,2}$, $b_{2,3}$ and $b_{2,4}$ by the formula \eqref{b2:s},
 \item $b_{3,1}$, $b_{3,2}$, $b_{3,3}$ and $b_{3,4}$ by the formula \eqref{b3:s}
\end{itemize}
then  $H_1({\Cal{M}}(N_{3,2});H_1(N_{3,2};\zz))$ is generated by homology classes
\[a_{1,3},\ a_{1,4},\ a_{1,1}+a_{1,2},\  u_3,\ u_4,\ d_{1,1},\]
with respect to the relations obtained in the formulas \eqref{a1:mod2}, \eqref{u34:mod2} and \eqref{d11:mod2}
\[2a_{1,3}=2a_{1,4}=2\left(a_{1,1}+a_{1,2}\right)=2u_3=2u_4=2d_{1,1}=0.\]
\end{proof}
This concludes the proof of Theorem \ref{MainThm1} which is an immediate consequence of Proposition \ref{prop:h1:1}.
\section{Computing $H_1({\Cal{PM}^+}(N_{3}^2);H_1(N_{3}^2;\zz))$}
By Theorems \ref{tw:pres:two:holes} and \ref{tw:pres:two:punct}, the presentation of the group 
${\Cal{PM}^+}(N_{3}^2)$ can be obtained from the presentation of ${\Cal{M}}(N_{3,2})$ by adding relations $d_1=d_2=1$.
Hence, it is straightforward to conclude from Proposition \ref{prop:kernel:2} and the proof of Proposition \ref{prop:h1:1} that
\begin{prop}\label{prop:h1:2}  The abelian group $H_1({\Cal{PM}^+}(N_{3}^2);H_1(N_{3}^2;\zz))$ has a presentation with generators
\[ a_{1,3},\ a_{1,4},\ a_{1,1}+a_{1,2},\ u_3,\ u_4,\]
and relations
\[2a_{1,3}=2a_{1,4}=2\left(a_{1,1}+a_{1,2}\right)=2u_3=2u_4=0.\qed\]
\end{prop}
This concludes the proof of Theorem \ref{MainThm2}.


\section*{Acknowledgements}
The author wishes to thank the referee for his/her helpful suggestions.


\begin{thebibliography}{1}

\bibitem{PawlakStukowHomoPunctured}
{\sc P.~Pawlak and M.~{S}tukow}, { The first homology group with twisted
  coefficients for the mapping class group of a non--orientable surface with
  boundary}.
\newblock arXiv:2206.13642, 2022.

\bibitem{Stukow_homolTopApp}
{\sc M.~Stukow}, { The first homology group of the mapping class group of a
  nonorientable surface with twisted coefficients}, Topology Appl. {\bf 178} (2014),
  417--437.

\bibitem{Stukow_HiperOsaka}
{\sc M.~Stukow}, { A finite presentation for the hyperelliptic mapping class group of a nonrientable
  surface}, Osaka J. Math. {\bf 52} (2015), 495--515.

\bibitem{Szep_curv}
{\sc B.~Szepietowski}, { A presentation for the mapping class group of a
  non-orientable surface from the action on the complex of curves}, Osaka J.
  Math. {\bf 45} (2008), 283--326.

\end{thebibliography}

\end{document}